\documentclass[10pt,a4paper]{article}
\usepackage{latexsym,amsthm,amssymb,amscd,amsmath}
\newcounter{alphthm}
\theoremstyle{plain}
\newtheorem{thm}{Theorem}[section]

 \newtheorem{exam}[thm]{Example}
 
 \theoremstyle{definition}
 \newtheorem{defn}[thm]{Definition}
 \theoremstyle{remark}
 \newtheorem{rem}[thm]{\bf Remark}
 \numberwithin{equation}{subsection}

\begin{document}

\title {Fixed-Point Theorem For Mappings Satisfying
a General Contractive Condition Of Integral Type Depended an
Another Function  }

\author{ S. Moradi\\\\
Faculty  of Science, Department of Mathematics\\
Arak University, Arak,  Iran\\
\date{}
 }
 \maketitle
\begin{abstract}
We established a fixed-point theorem for mapping satisfying a
general contractive inequality of integral type depended an
another function. This theorem substantially extend the theorem
due to Branciari (2003) and Rhoades (2003).
\end{abstract}


\section{Introduction}

In 2002 [2], Branciari established the Banach Contractive
Principle in the following theorem.

\begin{thm}
 Let $(X,d)$ be a complete metric space, $ k \in [0,1)$ and
$S:X \longrightarrow X$ be a mapping such that, for each $x,y \in
X$,
\[\int_0^{d(Sx,Sy)} \phi(t)  dt \leq k{\int _0^{d(x,y)}
\phi(t)  dt},  \qquad\qquad\qquad (1)\]

where $\phi: [0,+\infty)\longrightarrow [0,+\infty)$ is a
Lebesgue-integrable mapping which is summable (i.e., with finite
integral) on each compact subset of $[0,+\infty)$, nonnegative,
and such that for each $\epsilon > 0 ,\int _0^\epsilon \phi(t) dt
> 0$; then $S$ has a unique fixed point $b\in X$ such that for
each $x\in X$, $\underset {n\rightarrow\infty} \lim S^{n} x=b$.
\end{thm}

After this result in (2003), Rhoades established the Branciari
Theorem in the following.

\begin{thm}
 Let $(X,d)$ be a complete metric space, $ k \in [0,1)$ and
$S:X \longrightarrow X$ a mapping such that, for each $x,y \in X$,
\[\int_0^{d(Sx,Sy)} \phi(t)  dt \leq k{\int _0^{m(x,y)}
\phi(t)  dt},  \qquad\qquad\qquad (2)\]
where
\[m(x,y)=\max \{ d(x,y),d(x,Sx),d(y,Sy),\frac{d(x,Sy)+d(y,Sx)}{2} \} \qquad\qquad (3)\]
and $\phi: [0,+\infty)\longrightarrow [0,+\infty)$ is a
Lebesgue-integrable mapping which is summable (i.e., with finite
integral) on each compact subset of $[0,+\infty)$, nonnegative,
and such that for each $\epsilon > 0 ,\int _0^\epsilon \phi(t) dt
> 0$. Then $S$ has a unique fixed point $b\in X$ such that for
each $x\in X$, $\underset {n\rightarrow\infty} \lim S^{n} x=b$.
\end{thm}

In 2009 [1]  A. Beiranvand, S. Moradi, M. Omid and H. Pazandeh
introduced a new class of contractive mapping and extend the
Banach Contractive Principle.

Also in 2009 [4] A. Beiranvand and S. Moradi established the
Branciari Theorem for these classes of mappings. It is the
purpose of this paper to make an extension the Rhoades Theorem
(Theorem 1.2).

For the main theorem (Theorem 2.1) we need the following
definition.

\begin{defn}$[1]$
 Let $(X,d)$ be a metric space. A mapping $T:X \longrightarrow X$
 is said sequentially convergent if we have, for every sequence
 $\{y_{n}\}$, if $\{Ty_{n}\}$ is convergence then $\{y_{n}\}$ also is
 convergence. $T$ is said subsequentially convergent if we have, for
 every sequence $\{y_{n}\}$, if $\{Ty_{n}\}$ is convergence then
 $\{y_{n}\}$ has a convergent subsequence.
\end{defn}


\section{Main Result}

The following theorem (Theorem 2.1) is the main result of this
paper.

\begin{thm}
 Let $(X,d)$ be a complete metric space, $ k \in [0,1)$ and
$S:X \longrightarrow X$ a mapping such that, for each $x,y \in X$,
\[\int_0^{d(TSx,TSy)} \phi(t)  dt \leq k{\int _0^{m'(Tx,Ty)}
\phi(t)  dt},  \qquad\qquad\qquad (4)\] where
\[m'(Tx,Ty)=\max \{ d(Tx,Ty),d(Tx,TSx),d(Ty,TSy),\frac{d(Tx,TSy)+d(Ty,TSx)}{2} \} \qquad (5)\]
and $\phi: [0,+\infty)\longrightarrow [0,+\infty)$ is a
Lebesgue-integrable mapping which is summable (i.e., with finite
integral) on each compact subset of $[0,+\infty)$, nonnegative,
and such that

\[ for \: each \: \epsilon > 0 \qquad \int _0^\epsilon \phi(t) dt
> 0  \qquad\qquad\qquad (6)\]
and $T:X \longrightarrow X$ is a continuous, one-to-one and
subsequentially convergent. Then $S$ has a unique fixed point
$b\in X$ and, if $T$ is sequentially convergent then for each
$x\in X$, $\underset {n\rightarrow\infty} \lim S^{n} x=b$.
\end{thm}
\begin{proof}
From (4) $S$ is continuous and if $x \neq y$ then,
\[d(TSx,TSy) < m'(x,y). \qquad\qquad\qquad\qquad (7)\]
Let $x \in X$. Define $x_{n}=TS^{n}x$. From (5) we conclude that:
\begin{eqnarray*}
  &&m'(x_{m},x_{n})= m'(TS^{m}x,TS^{n}x)=\\
  &&\max \{d(x_{m},x_{n}),d(x_{m},x_{m+1}),d(x_{n},x_{n+1}),
  \frac{d(x_{n},x_{m+1})+d(x_{m},x_{n+1})}{2}\}. \qquad\qquad(8)
\end{eqnarray*}
We break the argument into four steps.

\textbf{STEP 1}. $\underset{n \rightarrow \infty}{\lim}
d(x_{n},x_{n+1})=0$.\\
\textbf{proof.} For each integer $n \geq 1$, from (4),
\[\int_0^{d(x_{n},d_{n+1}) \phi(t)  dt} \leq k{\int _0^{m'(x_{n-1},x_{n})}
\phi(t)  dt},  \hspace {5cm} (9)\] and by (8),
\begin{eqnarray*}
  m'(x_{n-1},x_{n})
  &=&\max \{d(x_{n-1},x_{n}),d(x_{n-1},x_{n}),d(x_{n},x_{n+1}),
  \frac{d(x_{n-1},x_{n+1})+d(x_{n},x_{n})}{2}\} \\
  &=& \max \{ d(x_{n-1},x_{n}),d(x_{n},x_{n+1}),
  \frac{d(x_{n-1},x_{n+1})}{2} \} \\
  &\leq& \max \{d(x_{n-1},x_{n}),d(x_{n},x_{n+1}),
  \frac{d(x_{n-1},x_{n})+d(x_{n},x_{n+1})}{2}\} \\
  &=& \max \{ d(x_{n-1},x_{n}),d(x_{n},x_{n+1}) \} \:\:(from \: (6)
  \:and \: (7))\\
  &=&d(x_{n-1},x_{n}).\hspace {7cm}(10)
\end{eqnarray*}
Hence, by (9) and (10) we have,
\[\int_0^{d(x_{n},d_{n+1}) \phi(t)  dt} \leq k^{n}{\int _0^{d(x,x_{1})}
\phi(t)  dt}.  \hspace {3cm} (11)\] Taking the limit of (11), as
$n \rightarrow \infty$, gives $\underset{n \rightarrow
\infty}{\lim} \int_0^{d(x_{n},d_{n+1}) \phi(t) dt}=0$. Since (6)
is holds,
\[\underset{n \rightarrow \infty}{\lim} d(x_{n},x_{n+1})=0. \qquad\qquad\qquad\qquad\qquad\qquad\qquad (12)\]

\textbf{STEP 2}. $\{x_{n}\}$ is a bounded sequence.
\textbf{proof.} If $\{x_{n}\}$ is not a bounded sequence then, we
choose a sequence $\{n(k)\}_{k=1}^{\infty}$ such that $n(1)=1$
and for each $k \in \Bbb{N}$; $n(k+1)$ is "minimal" in the sense
such that $d(x_{n(k+1)},x_{n(k)}) > 1$. Obviously $n(k) \geq k$
for all $k \in \Bbb{N}$.\\
By step 1, there exists $k_{0} \in \Bbb{N}$ such that for every
$k \geq k_{0}$; $d(x_{k+1},x_{k})< \frac{1}{2}$. So for each $k
\geq k_{0}$;
\begin{eqnarray*}
1 < d(x_{n(k+1)},x_{n(k)}) &\leq&
d(x_{n(k+1)},x_{n(k+1)-1})+d(x_{n(k+1)-1},x_{n(k)})\\
&\leq& d(x_{n(k+1)},x_{n(k+1)-1})+1. \qquad\qquad\qquad\qquad (13)
\end{eqnarray*}
By (12) and (13) we conclude that,
\[ \underset{n \rightarrow \infty}{\lim} d(x_{n(k+1)},x_{n(k)})=1. \qquad\qquad\qquad\qquad\qquad\qquad(14)\]
Also,
\begin{eqnarray*}
&&d(x_{n(k+1)},x_{n(k)})-d(x_{n(k+1)+1},x_{n(k+1)})-d(x_{n(k)+1},x_{n(k)})\\
&&\leq d(x_{n(k+1)+1},x_{n(k)+1}) \leq
d(x_{n(k+1)+1},x_{n(k+1)})\\
&&+d(x_{n(k+1)},x_{n(k)})+d(x_{n(k)},x_{n(k)+1}).
 \hspace{2.5cm} (15)
\end{eqnarray*}
Since (12), (14) and (15) are hold,
\[ \underset{n \rightarrow \infty}{\lim} d(x_{n(k+1)+1},x_{n(k)+1})=1. \qquad\qquad\qquad\qquad\qquad\qquad(16)\]
Therefore by (8),
\begin{eqnarray*}
 &&m'(x_{n(k+1)},x_{n(k)})=\max \{
 d(x_{n(k+1)},x_{n(k)}),d(x_{n(k+1)},x_{n(k+1)+1}),\\
 &&d(x_{n(k)},x_{n(k)+1}),
\frac{d(x_{n(k)},x_{n(k+1)+1})+d(x_{n(k+1)},x_{n(k)+1})}{2}\},
\qquad\qquad(17)
\end{eqnarray*}
from (12) and (14), for large enough $k$,
\begin{eqnarray*}
 &&m'(x_{n(k+1)},x_{n(k)})=\max \{
 d(x_{n(k+1)},x_{n(k)}),\frac{d(x_{n(k)},x_{n(k+1)+1})+d(x_{n(k+1)},x_{n(k)+1})}{2}\}\\
 &&=\max \{
 d(x_{n(k+1)},x_{n(k)}),\frac{[d(x_{n(k)},x_{n(k)+1})+d(x_{n(k)+1},x_{n(k+1)+1})]}{2}+\\
 &&\frac{[d(x_{n(k+1)},x_{n(k+1)+1})+d(x_{n(k+1)+1},x_{n(k)+1})]}{2}\} \underset{k \rightarrow \infty}\longrightarrow
 1.
\qquad\qquad\qquad\qquad\qquad(18)
\end{eqnarray*}
So by (16) and (18) and
\[\int_{0}^{d(x_{n(k+1)+1},x_{n(k)+1})} \phi(t)dt \leq k \int_{0}^{m'(x_{n(k+1)},x_{n(k)}))}\phi(t)dt, \qquad\qquad\qquad(19)\]
we conclude that,
\[\int_{0}^{1} \phi(t)dt \leq k \int_{0}^{1}\phi(t)dt. \qquad\qquad\qquad\qquad\qquad\qquad(20)\]
Since $k \in [0,1)$, $\int_{0}^{1} \phi(t)dt=0$ and this is
contradiction with (6).

\textbf{STEP 3}. $\{x_{n}\}$ is a Cauchy sequence.\\
\textbf{proof.} For every $m,n \in \Bbb{N} (m>n)$ by (4)
\begin{eqnarray*}
&&\int_{0}^{d(x_{m},x_{n})} \phi(t)dt \leq
\int_{0}^{m'(x_{m-1},x_{n-1})} \phi(t)dt\\
&&=k\int_{0}^{\max\{
d(x_{m-1},x_{n-1}),d(x_{m-1},x_{m}),d(x_{n-1},x_{n}),\frac{d(x_{m-1},x_{n})+d(x_{n-1},x_{m})}{2}\}}\phi(t)dt\\
&& \leq k \int_{0}^{\max\{
d(x_{m-1},x_{n-1}),d(x_{m-1},x_{m}),d(x_{n-1},x_{n}),d(x_{m-1},x_{n}),d(x_{n-1},x_{m})\}}\phi(t)dt\\
&&=k \int_{0}^{d(x_{r(1)},x_{s(1)})} \phi(t)dt,
\qquad\qquad\qquad\qquad\qquad\qquad\qquad(21)
\end{eqnarray*}
where $s(1) \geq n-1$ and $r(1) > s(1)$.\\
By the same argument, there exist $r(2),s(2) \in \Bbb{N}$ such
that $r(2) > s(2)$ and $s(2) \geq s(1)-1 \geq n-2$ such that
\[ \int_{0}^{d(x_{r(1)},x_{s(1)})} \phi(t)dt \leq k \int_{0}^{d(x_{r(2)},x_{s(2)})} \phi(t)dt. \qquad\qquad\qquad\qquad\qquad(22)\]
So, by (21) and (22),
\[ \int_{0}^{d(x_{m},x_{n})} \phi(t)dt \leq k^{2} \int_{0}^{d(x_{r(2)},x_{s(2)})} \phi(t)dt. \qquad\qquad\qquad\qquad\qquad(23)\]
By the same argument, there exist $r(n),s(n) \in \Bbb{N}$ such
that $r(n) > s(n)$ and $s(n) \geq s(n)-n \geq n-n=0$ and
\[ \int_{0}^{d(x_{m},x_{n})} \phi(t)dt \leq k^{n} \int_{0}^{d(x_{r(n)},x_{s(n)})} \phi(t)dt. \qquad\qquad\qquad\qquad\qquad(24)\]
Since $\{x_{n}\}$ is a bounded sequence and (24) is holds,
\[\underset{m,n \rightarrow \infty}{\lim} \int_{0}^{d(x_{m},x_{n})}\phi(t)dt=0. \qquad\qquad\qquad\qquad\qquad(25)\]
Hence, from (6),
\[\underset{m,n \rightarrow \infty}{\lim} d(x_{m},x_{n})=0. \qquad\qquad\qquad\qquad\qquad\qquad(26)\]
Therefore $\{x_{n}\}$ is a Cauchy sequence.

\textbf{Step 4}. $S$ has a fixed point.\\
\textbf{proof.} Since $(X,d)$ is a complete metric space and
$\{x_{n}\}$ is a Cauchy sequence there exists $a \in X$ such that
\[ \underset{n \rightarrow \infty}{\lim} TS^{n}(x)=a. \hspace{5cm}(27)\]
Since $T$ is subsequentially convergent, $\{S^{n}(x)\}$ has a
convergent subsequence alike $\{S^{n(k)}(x)\}_{k=1}^{\infty}$.
Suppose that
\[ \underset{k \rightarrow \infty}{\lim} S^{n(k)}(x)=b. \hspace{5cm}(28)\]
Since $T$ is continuous,
\[ \underset{k \rightarrow \infty}{\lim} TS^{n(k)}(x)=Tb. \hspace{4.5cm}(29)\]
From (27) and (29) we conclude that
\[Tb=a. \hspace{6.5cm}(30)\]
Since $S$ is continuous and (28) is holds,
\[ \underset{k \rightarrow \infty}{\lim} S^{n(k)+1}(x)=Sb. \hspace{4.5cm}(31)\]
So,
\[ \underset{k \rightarrow \infty}{\lim} TS^{n(k)+1}(x)=TSb. \hspace{4cm}(32)\]
Again from (27) and (30)
\[TSb=a=Tb. \hspace{5.5cm}(33)\]
Since $T$ is one-to-one, $Sb=b$. Therefore $S$ has a fixed
point.\\
Obviously, by (4) and (6) we conclude that $S$ has a unique fixed
point.
\end{proof}
\begin{rem}
Theorem 2.1 is a generalization of the Rhoades theorem (Theorem
1.2), letting $Tx=x$ for each $x\in X$ in Theorem 2.5, so \\
\begin{eqnarray*}
\int_0^{d(Sx , Sy)}\phi(t) dt
 &=& \int_0^{d(TSx , TSy)} \phi(t) dt\\
 &\leq&  k {\int_0^{m'(x,y)} \phi(t) dt }=k {\int_0^{m(x,y)} \phi(t) dt }.\qquad(34)
\end{eqnarray*}
\end{rem}

The following example shows that (4) is indeed a proper extension
of (2).
\begin{exam}
Let $X=[1,+\infty)$ endowed with the Euclidean metric. Define
$S:X \longrightarrow X$ by $Sx=4\sqrt{x}$. Obviously $S$ has a
unique fixed point $b=16$.\\
If (2) holds for some $k \in [0,1)$, then for every $x,y \in X$
such that $x \neq y$, we have
\[|Sx-Sy| < m(x,y). \hspace{5.5cm}(35)\]
But by taking $x=1$ and $y=4$ we have, $|Sx-Sy|=m(x,y)=4$ and
this is contradiction. Therefore we can not use the Rhoades
theorem (Theorem 1.2) for this example.\\
Now we define $T:X \longrightarrow X$ by $Tx=\ln(e.x)$. Obviously
$T$ is one-to-one, continuous and sequentially convergent and
\[ |TSx -TSy| =\frac{1}{2}|\ln(\frac{e.x}{e.y})|=\frac{1}{2}|Tx-Ty| \leq \frac{1}{2}m'(Tx,Ty). \qquad\qquad(36)\]
By taking $\phi \equiv 1$, all conditions of Theorem 2.1 are hold
and therefore $S$ has a unique fixed point.
\end{exam}


Email:

S-Moradi@araku.ac.ir

\end{document}